\documentclass[12pt]{amsart}
\usepackage{amscd,amssymb,graphics}

\usepackage{amsfonts}
\usepackage{amsmath}
\usepackage{amsxtra}
\usepackage{latexsym}
\usepackage[mathcal]{eucal}

\usepackage{graphics,colortbl}

\input xy
\xyoption{all}
\usepackage{epsfig}

\usepackage[pdftex,bookmarks,colorlinks,breaklinks]{hyperref}

\newtheorem{theorem}{Theorem}[section]
\newtheorem{fact}[theorem]{Fact}
\newtheorem{corollary}[theorem]{Corollary}
\newtheorem{lemma}[theorem]{Lemma}
\newtheorem{proposition}[theorem]{Proposition}

\newtheorem{question}[theorem]{Question}
\newtheorem{definition}[theorem]{Definition}
\numberwithin{equation}{section}
\newtheorem{claim}{Claim} 

\theoremstyle{remark}
\newtheorem{remark}[theorem]{Remark}
\newtheorem{example}[theorem]{Example}

\newcommand{\ben}{\begin{enumerate}}
\newcommand{\een}{\end{enumerate}}
\newcommand{\bit}{\begin{itemize}}
\newcommand{\eit}{\end{itemize}}

\def\R {{\Bbb R}}
\def\Q {{\Bbb Q}}
\def \F {{\Bbb F}}

\def\C {{\Bbb C}}
\def\N{{\Bbb N}}
\def\T{{\Bbb T}}

\def\Z {{\Bbb Z}}

\def\F{{\Bbb F}}

\def\H{{\mathcal H}}
\def\eps{{\varepsilon}}

\def\QED{\nobreak\quad\ifmmode\roman{Q.E.D.}\else{\rm Q.E.D.}\fi}

\def\a {\alpha}

\def\PGL{\operatorname{PGL}}
\def\PSL{\operatorname{PSL}}
\def\GL{\operatorname{GL}}
\def\SL{\operatorname{SL}}
\def\UT{\operatorname{UI}}
\def\UT{\operatorname{UT}}

\def\SO{\operatorname{SO}}
\def\SU{\operatorname{SU}}
\def\H{\operatorname{H}}
\def\O{\operatorname{O}}

\def\STP{\operatorname{ST^+}}
\def\TP{\operatorname{T^+}}
\def\STM{\operatorname{ST^{--}}}
\def\TM{\operatorname{T^{--}}}
\begin{document}

\title[]{Minimality of topological matrix groups and Fermat primes}

\dedicatory{Dedicated to Prof. D. Dikranjan on the occasion of
	his 70th birthday}

		\author[]{M. Megrelishvili}
		\address[M. Megrelishvili]
	{\hfill\break Department of Mathematics
		\hfill\break
		Bar-Ilan University, 52900 Ramat-Gan	
		\hfill\break
		Israel}
		\email{megereli@math.biu.ac.il}

\author[]{M. Shlossberg}
\address[M. Shlossberg]
{\hfill\break School of Computer Science
	\hfill\break Reichman University, 4610101 Herzliya 
	\hfill\break Israel}
\email{menachem.shlossberg@post.idc.ac.il}



 \date{2022, August}  
\thanks{This research was supported by a grant of the Israel Science Foundation (ISF 1194/19) 
	and also by the Gelbart Research Institute at the Department of Mathematics, Bar-Ilan  University} 
  
\begin{abstract}  	 
Our aim is to study topological minimality of some natural matrix groups. We show  that the special upper triangular group $\STP(n,\F)$ is minimal for every local field $\F$ of characteristic $\neq 2$. This result is new even for the field $\R$ of reals and it leads to some important consequences. We prove criteria for the minimality and  total minimality of the special linear group $\SL(n,\F)$, where $\F$ is a subfield of a local field. This extends some known results of Remus--Stoyanov (1991) and Bader--Gelander (2017).  

 One of our main applications is a characterization of Fermat primes, which asserts that for an odd prime  $p$ the following conditions are equivalent: 
 \begin{enumerate} 
 	\item $p$ is a  Fermat prime;
 	\item $\SL(p-1,\Q)$ is minimal, where $\Q$ is the field of rationals equipped with the $p$-adic topology;
 	\item  $\SL(p-1,\Q(i))$ is minimal, where $\Q(i) \subset \C$ is the Gaussian rational field.	
 \end{enumerate} 
\end{abstract}

\maketitle
	
\setcounter{tocdepth}{1}
	\tableofcontents

	\section{Introduction} 

It is a well-known phenomenon that many natural topological groups in analysis and geometry are minimal \cite{DPS89,Stoyanov,RS,Mayer,Usp, Gl, DM14, BT, Duchesne}. 
 For a survey, regarding minimality in topological groups, we refer 
to \cite{DM14}. 
		
		All topological spaces in the sequel are Hausdorff.
		A  
		topological group $G$ is  \emph{minimal} \cite{Doitch,S71} if 
		every continuous isomorphism $f \colon G\to H$, with $H$ 
		a  topological group, is a topological isomorphism 
		(equivalently, if $G$ does not admit a strictly coarser group topology). 
			If every quotient of $G$ is minimal, then $G$ is called  \emph{totally minimal} \cite{DP}.  
			Recall also \cite{MEG04,DM10} that a subgroup $H$ of $G$ is said to be \emph{relatively minimal} (resp., \emph{co-minimal}) in $G$ if every coarser 
		group topology on $G$ induces on $H$ (resp., on the coset set $G/H$) the original topology. 
	 
	 \vskip 0.2cm 
	 Let $\F$ be a topological field. Denote by $\GL(n,\F)$  the group of $n\times n$ invertible matrices over the field $\F$  with the natural pointwise topology 
	 inherited from $\F^{n^2}$. Consider the following topological subgroups of $\GL(n,\F)$:  
	 \bit 
	 \item $\SL(n,\F)$ -- Special Linear Group -- matrices with determinant equal to $1$. 
	 \item $\TP(n,\F)$ -- Upper Triangular invertible  matrices.
	 \item $\STP(n,\F):=\TP(n,\F) \cap \SL(n,\F)$ --  Special Upper Triangular group. 
	 \item $\mathsf{N}:=\UT(n,\F)$--Upper unitriangular matrices. 
	 \item $\mathsf{D}$ -- Diagonal invertible matrices. 
     \item $\mathsf{A}:=\mathsf{D} \cap \SL(n,\F)$. Note that $\mathsf{N}\mathrm{A}=\STP(n,\F)$.
	 \eit  
	Consider also the following projective linear groups (equipped with the quotient topology):
	\bit
	  \item $\PGL(n,\F)= \GL(n,\F)/Z(\GL(n,\F)).$ 
	 \item $\PSL(n,\F)=  \SL(n,\F)/Z(\SL(n,\F)).$  
\eit
	 \vskip 0.2cm  
 \subsection{Main results}
 
 In this paper, we study minimality conditions in 
 topological  matrix groups over local fields and their subfields. 
 We thank D. Dikranjan whose kind suggestions led us to the following question  which hopefully opens several fruitful research lines. 
 
 \begin{question} Let $G$ be a subgroup of $\GL(n,\F)$.  Under which conditions 
 	is $G$ (totally) minimal? 
 \end{question} 

 We prove  in Theorem \ref{thm:sutmin} that the solvable group  $\STP(n,\F)$ is minimal for every local field $\F$ of characteristic distinct from $2$ and every $n\in\N$. 
   This result is new even for the field $\R$ of reals. 

 Using Iwasawa decomposition, 
 our results on $\STP(n,\F)$ lead 
  (see Theorem \ref{pro:pslntm}) to the total minimality of $\SL(n,\F)$ 
 for local fields $\F$ of characteristic distinct from~$2$.
 According to an important result of Remus and Stoyanov \cite{RS},   
 $\SL(n,\R)$ is totally minimal. More generally, a connected semi-simple Lie group is totally minimal if and only if its center is finite.  
 Recent results of Bader and Gelander \cite{BG}, obtained in a different way, imply that $\SL(n,\F)$ is totally minimal for every 
 local field $\F$ with any characteristic. 

We provide criteria  for the minimality
 and  total minimality of $\SL(n,\F)$, where $\F$ is a  subfield of a local field 
 (see Theorem \ref{t:SLSMALL} and Proposition~\ref{pro:minsl}).  Corollary \ref{cor:tmsln}(1) shows that $\SL(2,\F)$ is totally minimal,  while  
   $\SL(2^k,\F)$ is minimal for every $k\in \N$, by Corollary 
\ref{cor:p2}. It also turns out that $\SL(n,\F)$ is totally minimal for every topological subfield $\F$ of $\R$ (see Corollary \ref{cor:tmsln}(2)). 

  Sometimes for the same field, according to the parameter $n \in \N$, we have all three possibilities: minimality, total minimality  and the absence of minimality. Indeed, see Corollary \ref{ex:slnotmin} which gives a trichotomy concerning the group  $\SL(n,\Q(i))$,  where $\Q(i):=\{a+bi: a,b\in\Q\}$ is the Gaussian rational field. 
  If $n$ is not a power of $2,$ then $\STP(n,\Q(i))$ and $\SL(n,\Q(i))$ are not minimal. 

 By Corollary \ref{c:p-adic}, 
 	if $p-1$ is not a power of $2,$
 	then $\SL(p-1,(\Q,\tau_p))$ is not minimal, 
 	where $(\Q,\tau_p)$ is the field of rationals with the $p$-adic topology (treating it as a subfield of the local field $\Q_p$ of all $p$-adic numbers). 
 	Furthermore, for every subfield $\F$ of $\Q_p,$ the groups $\SL(n,\F)$ and $\PSL(n,\F)$ are totally minimal for every $n$ which is coprime to $p-1.$ 
 	
  It is known that if $p=2^k+1$ is an odd prime then $k$ is a power of $2$. 
 These are the famous \emph{Fermat primes} 
 $F_n=2^{2^n}+1.$ As of 2021, the only known Fermat primes are $F_0 = 3, \ F_1 = 5,\ F_2 = 17, \ F_3 = 257,$ and $F_4 = 65537.$  
The following theorem is  one of our main 
applications (proved in Theorem \ref{thm:fermat}), 
which characterizes Fermat primes in terms of topological minimality.  
\begin{theorem} \label{t:MainIntro} 
	For an odd prime  $p$ the following conditions are equivalent: \ben\item $p$ is a  Fermat prime;
	\item $\SL(p-1,(\Q,\tau_p))$ is minimal;
	\item  $\SL(p-1,\Q(i))$ is minimal.
	\een
\end{theorem} 
We prove in Theorem \ref{thm:ptm} that the projective general linear group 
$\PGL(n,\F)$ is totally minimal for every local field $\F$ and every $n\in \N$.  
 The same holds for topological subfields $\F$  of $\R$ as long as $n$ is  odd (see Theorem \ref{pro:sofr}).  

 
 %

\vskip 0.5cm 
	 \subsection{Some known results} 
	 One of the first 
	 examples (due to Dierolf and Schwanengel \cite{DS79}) of a minimal locally compact group which is not totally minimal is 
	 \begin{align*}
	 	\R \rtimes \R_{+} & \cong \left\{\left( \begin{array}{cc}
	 		a & b  \\
	 		0 & 1   
	 	\end{array}  \right) \middle|  \ a \in \R_{+}, \  b \in \R \right\}.
	 \end{align*} 
	 
	Compact  groups are totally minimal. 
	 Minimal abelian groups  are necessarily precompact by a theorem of Prodanov--Stoyanov \cite{PS}. An interesting and useful generalization of this classical result has been found by T. Banakh \cite{Banakh}.  
	 For every minimal group $G$ its 
	 center $Z(G)$ is precompact. So, if $G$ is, in addition, 
	 \textit{sup-complete} (i.e.,  complete with respect to its two-sided uniformity) then $Z(G)$ must be compact. For this reason, the group $\GL(n,\R)$ is not minimal. 
	 However, there are closed nonminimal subgroups  of $\GL(n,\R)$ with compact (even, trivial) center. Indeed, the rank-two  discrete free group $F_2$ is embedded into $\SL(2,\Z).$ 
	 Now recall that $F_2$, being residually finite, admits a precompact group topology. 
	 

	 The minimality of Lie groups has been studied by many authors. Among others, we refer to van Est \cite{Est}, Omori \cite{Omori}, Goto \cite{Goto}, Remus--Stoyanov \cite{RS} and the references therein.  
	  By Omori \cite{Omori}, connected nilpotent Lie groups with compact center are minimal. 
	 In particular, the classical Weyl--Heisenberg group $(\T \oplus \R) \rtimes_{\a} \R$ is minimal, where $\T=\R/ \Z$.
	Moreover, as it was proved 
	in 
	 \cite{DM10}, the 
	 \emph{Generalized Weyl--Heisenberg groups} 
	 $H_0(V)=(\T \oplus V) \rtimes V^*$, defined for every 
	 Banach space $V$, are minimal. 
	 
	 
	 The affine groups $\R^n \rtimes \GL(n,\R)$ are minimal (Remus--Stoyanov \cite{RS}). 
	   Every closed matrix subgroup $G \leq \GL(n,\R)$ is a retract of a minimal Lie group of dimension $2n+1 + \dim(G)$. 
	 For every locally compact abelian group $G$ and its dual $G^*,$ the generalized Heisenberg group $(\T \oplus G^*) \rtimes G$ is minimal (see \cite{MEG95}). Therefore, every locally compact abelian group is a group retract of a locally compact minimal group. By \cite{MEG08}, every topological group is a group retract of a minimal group.

	 By a result of Mayer \cite{Mayer}, a locally compact connected group is totally minimal if and only if for every closed normal subgroup $N$ of $G$ the center $Z(G/N)$ is compact. In addition to $\SL(n,\R)$,  
	  the following concrete classical groups are totally minimal: the Euclidean motion group $\R^n \rtimes \SO(n,\R)$ and
	 the Lorentz group $\R^n \rtimes \SL(n,\R)$. 

 The unitary group of an infinite-dimensional Hilbert space is 
one of the most influential examples of a totally minimal group (see Stoyanov \cite{Stoyanov}).
By a result of Duchesne \cite{Duchesne},
the isometry group of the infinite dimensional separable hyperbolic
space 
is minimal. 
The locally compact solvable groups having all subgroups minimal were characterized recently in \cite{XDSD}.

	 \vskip 0.1cm  
	\noindent \textbf{Acknowledgment.}  We thank the referee for the valuable report.  We also thank U. Bader, D. Dikranjan, A. Elashvili,\ B. Kunyavskii and G. Soifer for their useful suggestions.  
	After reading a preliminary version of this work, Dikranjan pointed out 
	the connection between our results and Fermat primes. 
	      
	 \vskip 0.3cm 
	 \section{Preliminaries}

	 A subset 
	 $B \subseteq \F$ of a topological field $\F$ is \textit{bounded} if for every neighborhood $U(0)$ there exists a neighborhood $V(0)$ such that $VB  \subseteq U$. 
	 A subset $U$ of $\F$ that contains zero is \emph{retrobounded} if $(\F \setminus U)^{-1}$ is bounded.

If retrobounded neighborhoods of zero 
	form a fundamental system of neighborhoods, then $\F$ is said to be \emph{locally retrobounded}. 
	 It is equivalent (see \cite[Theorem~19.12]{Warner-R}) to say  that all neighborhoods of zero are retrobounded.

	\begin{remark} \label{r:strict}  
	The completion $\widehat\F$ of a locally retrobounded field $\F$ is again a locally retrobounded field 
			(see \cite[Theorems 13.9 and 8.3]{Warner-R}). 
			In general, the completion of a topological field is only a commutative topological ring and not always a field 
			(see \cite[p. 439]{Warner-R}).
	\end{remark}   

	Following Nachbin, a topological field $\F$ is said to be \emph{strictly minimal} (or,    \emph{straight}, \cite{Warner-R}) if $\F$ is a minimal $\F$-module over $\F$. Any non-discrete \emph{locally retrobounded field} 
	$K$ is strictly minimal. It is still unknown if any strictly minimal topological field is necessarily locally retrobounded (see \cite[p. 487]{Warner-R}). 
By \cite{MEG98}, a topological field $\F$ is strictly minimal if and only if the semidirect product $\F \rtimes \F^{\times}$ is a minimal topological group. 
Compare with the case of the group $\R \rtimes \R_{+}$ (Dierolf and Schwanengel) mentioned above.  
For 
another 
similar result, see Theorem \ref{t1} below.




\vskip 0.2cm  
A topological field is locally retrobounded if, for example, it is linearly ordered  or topologized by an absolute value.

\begin{definition} \rm{(see, for example, \cite[p. 26]{MA})}  
	A	\emph{local field} is a 
	non-discrete locally compact field.
\end{definition}
Every local field $\F$ admits an absolute value (induced by the Haar measure). 
Therefore, any subfield of a local field is locally retrobounded. 
If the set $\{|n\cdot 1_\F|: n\in \N\}$ is unbounded, then $\F$ is called \emph{archimedean}. Otherwise, $\F$ is a \emph{non-archimedean} local field (see \cite{V}).  
A subset of a local field is compact if and only if it is closed and bounded.	

	\subsection{Roots of unity} 	
Denote by $\mu_n(\F)$ the finite subgroup of $\F^{\times}$ consisting of all $n$-th roots of unity. Then $\SL(n,\F)$ has finite center (e.g., see \cite[3.2.6]{RO}) 
 $$Z=Z(\SL(n,\F))=\{\lambda I:\lambda\in \mu_n\}$$  
which, sometimes, will be denoted in Sections \ref{s:3} and \ref{s:4} simply by $Z$.
\newline 

The following known lemma will be used in the sequel. We prove it  for the sake of completeness.

\begin{lemma}\label{lem:rinqi}
	If $z^n=1$ and $z\in \Q(i)$, then $z\in\{\pm 1,\pm i\}$.
\end{lemma}
\begin{proof}
	If $z^n=1$, then $(\bar z)^{n}=1,$  where $\bar z$ is the complex conjugate of $z$. It follows that both $z$ and $\bar z$ are algebraic integers.
	By \cite[Proposition 6.1.5]{IR}, $z+\bar z$ is an  algebraic integer. Since $z\in \Q(i)$, the algebraic integer  $z+\bar z$ is also rational. By
	\cite[Proposition 6.1.1]{IR}, $z+\bar z$ is an integer. As $|z|=1$ and $z\in \Q(i),$ we deduce that $z\in\{\pm 1,\pm i\}$.
\end{proof}	

The following result about the simplicity of $\PSL(n,\F)= \SL(n,\F)/Z(\SL(n,\F))$ is due to Jordan and Dickson (see \cite[3.2.9]{RO}). 
\begin{fact} \label{fac:simple}
	Let $\F$ be a field. If either $n>2$ or $n=2$ and $|\F|>3$, then $\PSL(n,\F)$ is algebraically simple.
\end{fact}

\vskip 0.3cm 
\subsection{$G$-minimality and semidirect products} 
The following result is known as \textit{Merson's Lemma} (\cite[Lemma 7.2.3]{DPS89} or \cite[Lemma 4.4]{DM14}).
\begin{fact}\label{firMer}
	Let $(G,\gamma)$ be a (not necessarily Hausdorff) topological group and $H$ be a subgroup of $G.$ If $\gamma_1\subseteq \gamma$ is a coarser group topology on $G$ such that $\gamma_1|_{H}=\gamma|_{H}$ and $\gamma_1/H=\gamma/H$, then $\gamma_1=\gamma.$
\end{fact}
As a corollary, 
 one has:
\begin{fact} \label{MER}\cite[Corollary 3.2]{DM10}
	A topological group $G$ is minimal if and only if it contains a subgroup $H$ which is both relatively minimal and co-minimal in $G$.
\end{fact}

By $X \rtimes_{\pi} G,$ we mean 
the (topological) semidirect product of the (topological) groups $X,G$, where $\pi \colon G \times X \to X$ is a given (continuous) action by group automorphisms. We denote by $\F^{\times}$ the multiplicative group $\F\setminus\{0\}$. 
Given a semidirect product  $\F\rtimes_{\alpha} \F^{\times},$ we identify $\F$ with $\F\rtimes_{\alpha} \{1\} $ and $\F^{\times}$ with   $\{0\}\rtimes_{\alpha} \F^{\times}.$ 

If a topological group $G$ continuously acts on a topological group $X$ by group automorphisms, then $X$ is called a $G$-\textit{group}. Assuming that the $G$-group $X$ has no strictly coarser Hausdorff group topology such that the action of $G$ on $X$ remains continuous, then $X$ is  $G$-\textit{minimal}. 

\begin{fact} \cite[Proposition 4.4]{DM10} \label{f:rel-min=G-min}
	Let $(G,\sigma)$ be a topological group and $(X,\tau)$ be a $G$-group. The following are equivalent: \ben
	\item $X$ is $G$-minimal.
	\item $X$ is relatively minimal in the topological semidirect product $M:=(X \rtimes G, \gamma)$.
	\een
\end{fact}

%
%
%
%
%


\vskip 0.3cm 
\subsection{$\STP(n,\F)$ as a topological semidirect product}
Recall the following topological matrix group  
\begin{equation*}
\STP(n,\F):=\left\{\bar x= (x_{ij})| \ x_{ij} \in \F, x_{ij}=0 \ \forall i >j, \  \prod_{i=1}^n x_{ii}=1 \right \},
\end{equation*}   and its topological subgroups
\begin{equation*}
	\mathsf{N}:=\left \{\bar x| \  x_{ii}=1 \ \forall i \right \},
\end{equation*}  
\begin{equation*}
\mathsf{A}:=\left\{\bar x| \  x_{ij}=0 \ \forall i \neq j, \  \prod_{i=1}^n x_{ii}=1 \right\}.
\end{equation*}  
	\begin{lemma} \label{l:topSemProd} 
		$\STP(n,\F)$ is a topological semidirect product of the subgroups $\mathsf{N}$ and $\mathsf{A}$. That is,  $\STP(n,\F)\cong \mathsf{N} \rtimes_{\alpha} \mathsf{A}$, where $\alpha$ is the action by conjugations. 
	\end{lemma}
	\begin{proof} 
	 Clearly, $G:=\STP(n,\F)=\mathsf{NA}$ and  $\mathsf{N}\cap \mathsf{A}$ is trivial. Moreover, 
$\mathsf{N}$ is a closed normal subgroup in $G$ and $\mathsf{A}$ is a closed subgroup in $G$. 
	So, algebraically, $G$ is isomorphic to the semidirect product $\mathsf{N}\rtimes_{\alpha} \mathsf{A},$ where $\alpha$ is the action by conjugations. The corresponding isomorphism is the map 
		$$i \colon \mathsf{N}\rtimes_{\alpha} \mathsf{A} \to \STP(n,\F) \ \ \  (\bar b,\bar a) \mapsto \bar b \cdot \bar a. 
	$$ 
	Observe that     $\bar c=\bar b \cdot \bar a$ satisfies  $c_{ij}=b_{ij}a_{jj}$ for every $i <j$ and $c_{ii}=a_{ii}$ for every $i.$
	
	Using the definition of the pointwise topology (and the fact that $\F$ is a topological field), it is easy to see that $i$ is a homeomorphism. 
\end{proof} 

\begin{remark} \label{r:standing} \
\begin{enumerate}
	\item Unless otherwise stated, below we assume that all fields are of characteristic distinct from $2$. 
	\item  By  $\widehat\F$ we always mean the completion of a locally retrobounded field $\F$ which always exists (Remark \ref{r:strict}(1)). If $\F$ is a subfield of a local field $P,$ then the 
	completion $\widehat\F$ can be identified with the closure of $\F$ in $P$. In case $\F$ is  infinite then  $\widehat\F$ is also a local field, as
	the local field $P$ contains no infinite discrete subfields (see \cite[p.  27]{MA}).
\end{enumerate}
\end{remark}

	\vskip 0.5cm  
	\section{Minimality of $\STP(n,\F)$} \label{s:3} 

	By Prodanov--Stoyanov's theorem,  if $\F$ is an infinite topological field, then neither $\mathsf{N}:=\UT(n,\F)$ nor $\TP(n,\F)$  are minimal as their centers are not precompact. 
In this section, we study the minimality of $\STP(n,\F).$  This case is also the key for further investigation.
	
	\subsection{Minimality of $\STP(2,\F)$} 
	\label{s:SUT} 
	
	Fixing $n=2$ in Lemma \ref{l:topSemProd}, we obtain the following subgroups of $\SL(2,\F)$ in a more explicit form:
	\begin{align*}
	\STP(2,\F) &=\left\{\left( \begin{array}{cc}
	a & b  \\
	0 & a^{-1}  
	\end{array}  \right) \middle|  \ a\in \F^{\times}, \  b\in \F \right\},\\
	\mathsf{A}&=\left\{\left( \begin{array}{cc}
	a & 0  \\
	0 & a^{-1}  
	\end{array}  \right) \middle|  \ a\in \F^{\times}\right\},\\
	\mathsf{N}&=\left\{\left( \begin{array}{cc}
	1 & b  \\
	0 & 1  
	\end{array}  \right) \middle|  \ b\in \F\right\}.
	\end{align*}

	\begin{lemma} \label{lem:vsq}
		The group $\STP(2,\F)$ is topologically isomorphic to the semidirect product $\F\rtimes_{\alpha} \F^{\times}$, where the action $\alpha: \F^{\times}\times \F \to \F$ is defined by $\alpha(a,b)=a^2b.$
	\end{lemma} 
	\begin{proof}
		 As we already know $G$ is topologically isomorphic to $\mathsf{N}\rtimes_{\beta} \mathsf{A}$, where $\beta$ is the action by conjugations.
		Explicitly, we have the following topological group isomorphism:  	
		$$\mathsf{N}\rtimes_{\beta} \mathsf{A} \to \STP(2,\F) \ \ \  (\bar b,\bar a) \mapsto \left(\begin{array}{cc} 
		a & \frac{b}{a}  \\
		0 & a^{-1}
			\end{array} \right), 
		$$  where $\bar a=\left(\begin{array}{cc}
			a & 0  \\
			0 & a^{-1}  
		\end{array}\right)$ and $\bar b=\left( \begin{array}{cc}
		1 & b  \\
		0 & 1  
		\end{array}\right).$
		Since $$\left(\begin{array}{cc}
		a & 0  \\
		0 & a^{-1}  
		\end{array}\right) \left( \begin{array}{cc}
		1 & b  \\
		0 & 1  
		\end{array}\right)\left(\begin{array}{cc}
		a & 0  \\
		0 & a^{-1}  
		\end{array}\right)^{-1}=\left(\begin{array}{cc}
		1 & a^2b  \\
		0 & 1  
		\end{array}\right),$$ it follows that $\mathsf{N}\rtimes_{\beta} \mathsf{A} \cong \F\rtimes_{\alpha} \F^{\times} $ which completes the proof.
	\end{proof}

%

	\vskip 0.3cm

\begin{proposition}\label{prop:retcom}
$\STP(2,\F)$ is a minimal topological group for every non-discrete  locally retrobounded complete field $\F.$  
\end{proposition}
	\begin{proof} 
			
		By Lemma \ref{lem:vsq}, it is equivalent to prove that  
		$\F\rtimes_{\alpha} \F^{\times}$ is minimal, where the action $\alpha \colon  \F^{\times}\times \F \to \F$ is defined by $\alpha(a,b)=a^2b.$
	We will show that the subgroup $\F$ is both  relatively minimal and co-minimal in $\F\rtimes_{\alpha} \F^{\times}.$
	This will prove the minimality of the latter by  Fact \ref{MER}.
	Denote by $\tau$ and $\tau^{\times}$ the given topologies on $\F$ and $\F^{\times},$ respectively, and let $\gamma$ be the product of these topologies on $\F\rtimes_{\alpha} \F^{\times}$.\\
	
	\vskip 0.2cm 
\noindent	{\bf Relative minimality} of $\F$

	To establish the relative minimality of $\F$  in $\F\rtimes_{\alpha} \F^{\times},$ it is equivalent to show by Fact \ref{f:rel-min=G-min}   that $\F$ is \textit{$\F^{\times}$-minimal}.
	For this purpose, let $\sigma\subseteq \tau$ be a coarser Hausdorff group topology on $\F$ such that 
	\begin{equation} \label{eq:alpha} 
	\alpha \colon \F^{\times}\times (\F,\sigma)\to (\F,\sigma) \ \ \ \alpha(a,b)=a^2b  
	\end{equation}
	remains continuous. We have to show that  $\sigma=\tau$.
	
	 Let $U$ be an arbitrary $\tau$-neighborhood of $0$. We will show that $U$ is a $\sigma$-neighborhood of $0$ and thus $\sigma=\tau.$   Since $\sigma$ is a Hausdorff group topology and the field $\F$ has characteristic distinct from $2$ (Remark \ref{r:standing}),   there exists a $\sigma$-neighborhood  $Y$ of $0$ such that $4\notin Y-Y.$ By the continuity of $\alpha$ and since $\F^{\times}$ is open in $\F,$  there exist a symmetric $\tau$-neighborhood $V$ of $0$ and  a $\sigma$-neighborhood $W$ of $0$ such that  
	 \begin{equation}\label{eq:fir}  \alpha((1+V)\times W)\subseteq Y.
	\end{equation}
	Since $\F$ is locally retrobounded, $U$ is retrobounded. That is, $(\F\setminus U)^{-1}$ is bounded in $(\F,\tau)$. So, there exists a $\tau$-neighborhood $M_1$ of zero in $\F$ such that $(\F\setminus U)^{-1} M_1 \subseteq V.$ Choose another $\tau$-neighborhood $M_2$ of zero such that $M_2 M_2 \subseteq M_1$. 
	Since $\F$ is not discrete, $M_2$ contains a nonzero element $\lambda$.  It follows that $$(\F\setminus U)^{-1}\lambda^2\subseteq V.$$

	By the continuity of $\alpha$ (see (\ref{eq:alpha})), we obtain that $\lambda^2W$ is a $\sigma$-neighborhood of $0.$ We claim that $\lambda^2W\subseteq U$ 
	(this will imply that $U$ is a $\sigma$-neighborhood of $0$ and  $\sigma=\tau$).
	Assume by contradiction that there exists $\mu\in W$ such that $\lambda^2\mu\notin  U.$ Then  $$\mu^{-1}=(\mu^{-1}\lambda^{-2})\lambda^2\in (\F\setminus U)^{-1}\lambda^2\subseteq V.$$ By (\ref{eq:fir}), we have $$\alpha(1+\mu^{-1},\mu)-\alpha(1-\mu^{-1},\mu)=(1+\mu^{-1})^2\mu-(1-\mu^{-1})^2\mu=4\in Y-Y,$$ a contradiction. 

	\vskip 0.3cm 
\noindent	{\bf Co-minimality} of $\F$
		
	
	Next, we prove that $\F$ is co-minimal in $(\F\rtimes_{\alpha} \F^{\times},\gamma).$ Let $\mu\subseteq \gamma$  be a coarser Hausdorff group topology on $\F\rtimes_{\alpha} \F^{\times}.$ 
We have to show that the coset topology $\mu / \F$ on $\F^{\times}$ is just the original topology $\tau^{\times}$ (which is equal to $\gamma / \F$). 
	It is equivalent to show that  the projection  $q\colon  (\F\rtimes_{\alpha} \F^{\times}, \mu) \to (\F^{\times}, \tau^{\times})$ is continuous. Since  $\gamma / \F=\tau^{\times},$ this will imply  that 
	$\gamma / \F= \mu/ \F,$ establishing the  co-minimality of $\F.$
	It suffices to show that the homomorphism $q$ is continuous at the identity $(0, 1).$ Let $U$ be a $\tau^{\times}$-neighborhood of $1.$ We will find  a $\mu$-neighborhood $V$ of
	$(0,1)$ such that $q(V) \subseteq U.$ Since $\F^{\times}$ is open in $(\F,\tau),$ it follows that there exists a $\tau$-neighborhood $O$ of $0$ such that $1+O\subseteq U.$ Being complete and relatively minimal in $\F\rtimes_{\alpha} \F^{\times},$ the subgroup $\F$ is also $\mu$-closed. Hence, the group topology $\mu/ \F$ is Hausdorff. Taking into account also the fact that $\mathrm{char}(\F)\neq 2,$ we find  $\mu/ \F$-neighborhoods $W_1,W_2$ of $1, -1,$ respectively, which are disjoint. Without loss of generality, there exists a $\mu$-neighborhood  $V_1$ of $(0,1)$ such that $q(V_1)=W_1.$ Using the fact that $\F^{\times}$ is open in $(\F,\tau)$ and since $\mu/ \F\subseteq \gamma/\F=\tau^{\times},$  we obtain that  $M=1+W_2$ is a $\tau$-neighborhood of $0.$     The definitions of $V_1$ and $M$ together with the fact that $W_1\cap W_2=\emptyset$ imply that
	\begin{equation}\label{eq:sec} q(V_1)+1\subseteq \F\setminus M.  \end{equation} 
	Since $\F$ is locally retrobounded, $(\F\setminus M)^{-1}$ is bounded. So, there exists a $\tau$-neighborhood $B$ of $0$ such that \begin{equation}\label{eq:third}(\F\setminus M)^{-1}B \subseteq O.\end{equation}  By the relative minimality of $\F,$ there exists  a  $\mu$-neighborhood $V_2$  of $(0,1)$ such that $V_2\cap \F=B.$ Since $\mu$ is a group topology, there exists a   $\mu$-neighborhood $V_3$   of $(0,1)$ such that the commutator $[(b,a),(1,1)]\in V_2$ for every $(b,a)\in V_3.$ Computing this commutator, we obtain  \begin{equation}\label{eq:four}
	[(b,a),(1,1)]= (b,a)(1,1)(b,a)^{-1}(1,1)^{-1}=(a^2-1,1)\in V_2\cap \F=B.
	\end{equation} 
	
	 Now we show that $q(V)\subseteq U$ for $V=V_1\cap V_3,$ which is a $\mu$-neighborhood of $(0,1).$ Fix an arbitrary $(b,a)\in V.$ By (\ref{eq:sec}) and since $V\subseteq V_1,$ we obtain 
	 $$(a+1)^{-1}\in (q(V)+1)^{-1}\subseteq (\F\setminus M)^{-1}.$$ Moreover, $V$ is also a subset of $V_3.$ So, (\ref{eq:four}) implies that $a^2-1\in B.$  Using (\ref{eq:third}), we now have
	  $$q(b,a)-1=a-1=(a+1)^{-1}(a^2-1)\in (\F\setminus M)^{-1}B \subseteq O.$$ 
	  Finally, we get $q(b,a)=1+(q(b,a)-1) \in 1 +O \subseteq U,$ as needed. 
	 
	  Now we can conclude that the topological group $\F\rtimes_{\alpha} \F^{\times}$ is minimal.\
\end{proof}

\vskip 0.3cm 
Let $H$ be a subgroup of a topological group $G$. Recall that $H$ is  {\it essential}  in $G$ if $H\cap L \neq \{e\}$ 
for every non-trivial closed normal subgroup $L$ of $G$. 
The following minimality criterion of dense subgroups is well-known (for compact $G$ see also \cite{P,S71}). 

\begin{fact} \label{Crit} \cite[Minimality Criterion]{B}  
	Let $H$ be a dense subgroup of a topological group $G.$  Then $H$ is minimal if and only if $G$ is minimal and $H$ is essential in $G$. 
\end{fact}


The following theorem deals with the minimality of $\STP(n,\F)$ only for $n=2$ in case $\F$ is a non-discrete locally retrobounded field. 
 However, if $\F$ is  a local field, then $\STP(n,\F)$ is minimal for every $n\in \N$ (see  Theorem \ref{thm:sutmin} below).

	\begin{theorem} \label{t1}
$\STP(2,\F)$ is minimal for every  non-discrete locally retrobounded field $\F.$ 
\end{theorem}

\begin{proof}
The completion  $\widehat\F$ of a locally retrobounded field $\F$ is a locally retrobounded field (Remark \ref{r:strict}(1)).
According to Lemma \ref{lem:vsq}, $\STP(2,\widehat{\F})$ is isomorphic to $\widehat{\F}\rtimes_{\alpha} (\widehat{\F})^{\times}$, where the action $\alpha \colon  (\widehat{\F})^{\times}\times \widehat{\F} \to \widehat{\F}$ is defined by $\alpha(a,b)=a^2b.$ Clearly, 
$\STP(2,\widehat{\F})$ contains $\STP(2,\F)\cong \F\rtimes_{\alpha} \F^{\times}$ as a dense subgroup.
  By Proposition \ref{prop:retcom}, $\widehat{\F}\rtimes_{\alpha} (\widehat{\F})^{\times}$ is minimal. To establish the minimality of $\STP(2,\widehat{\F})$ it is sufficient to prove, in view of Fact \ref{Crit}, that the subgroup  $\F\rtimes_{\alpha} \F^{\times}$ is essential in $\widehat{\F}\rtimes_{\alpha} (\widehat{\F})^{\times}.$ 	
  
  Let $L$ be a closed non-trivial normal subgroup of $\widehat{\F}\rtimes_{\alpha} (\widehat{\F})^{\times}$. We have to show that $L\cap (\F\rtimes_{\alpha} \F^{\times})$ is non-trivial. Let $(m,n)$ be a non-trivial element of $L$. If  $n\neq \pm 1$, then $1-n^2\neq 0$. Letting $a=(1-n^2)^{-1}$ and computing the commutator 
  $[(a,1),(m,n)],$ we obtain \[[(a,1),(m,n)]=(1,1)\in L\cap (\F\rtimes_{\alpha} \F^{\times}). \] Now assume that $n\in \{1,-1\}$. Since $(m,n)$ is non-trivial and   \[\{(0,-1),(-1,1),(-1,-1),(1,1),(1,-1)\}\subseteq \F\rtimes_{\alpha} \F^{\times},\] we may assume that $m\notin \{0,-1,1\}$. Moreover, without loss of generality, $n=1$. Indeed,  this follows from the fact that $(m,n)^{2}=(2m,1)$.   So, $(m,n)=(m,1)\in L,$ where $m\notin \{0,-1,1\}$. For every $a,b\in \F^{\times},$ we have \[(0,a)(m,1)(0,a)^{-1}(0,b)(m,1)^{-1}(0,b)^{-1}=((a^2-b^2)m,1)\in L,\] as $L$ is normal in $G$. In particular, letting $a=2^{-1} (1+m^{-1})$ and $b=a-1,$ we conclude that  \[((a^2-b^2)m,1)= ((a-b)(a+b)m,1)=(1,1)\in L\cap (\F\rtimes_{\alpha} \F^{\times}).\] This proves that  $\STP(2,\F)$ is essential in $\STP(2,\widehat{\F})$.
\end{proof}


\vskip 0.2cm 
Theorem \ref{t1} is not true for an arbitrary $n$. Indeed, 
in Example \ref{ex:nmin} below we prove
that $\STP(n,\Q(i))$ is not minimal in case $n$ is not a power of $2$. 

\begin{example}\label{ex:nmin}
Let $n$ be a natural number that is not a power of $2$. Then the  group $\STP(n,\Q(i))$ is not minimal.
	Indeed, by our assumption on $n,$ there exists an odd prime  $p$ that divides $n$. We claim that 
	the finite (hence, closed)  central subgroup $L=\{\lambda I:\lambda^p=1\}$ of $\STP(n,\C)$ trivially intersects $\STP(n,\Q(i))$.
To see this, observe that by Lemma \ref{lem:rinqi} if $\lambda^p=1$ and $\lambda\in \Q(i)$, then $\lambda=1.$
	This means that $\STP(n,\Q(i))$ is not essential in $\STP(n,\C)$. By the Minimality Criterion (Fact~\ref{Crit}), $\STP(n,\Q(i))$ is not minimal.
\end{example}

In view of Theorem \ref{t1} and Example \ref{ex:nmin}, the following natural questions arise:
\begin{question}
Let $k\in \N$ and $\F$ be a non-discrete locally retrobounded  field.  
Is $\STP(2^k,\F)$ minimal?
What if, in addition, $\F$ is complete?
\end{question}

\vskip 0.3cm 
\subsection{Minimality of $\STP(n,\F)$ and $\STP(n,\F)/Z(\SL(n,\F))$}  

Let $\F$ be a topological field. 
Recall that by Lemma \ref{l:topSemProd}, $\STP(n,\F)\cong \mathsf{N}\rtimes_{\alpha} \mathsf{A}$, where $\mathsf{N}=\UT(n,\F)$, 
  $\mathsf{A}$ is the group of diagonal matrices with determinant $1$  and $\alpha$ 
is the action 
by conjugations. In the sequel, we sometimes identify $\STP(n,\F)$ with $\mathsf{N}\rtimes_{\alpha} \mathsf{A}$. 

For $1\leq i<j\leq n,$ let $G_{i,j}$ be the 1-parameter subgroup of $\mathsf{N}$ such that for every matrix $X\in G_{i,j}$ we have $p_{k,l}(X)=x_{k,l}=0$ if $k\neq l$ and $(k,l)\neq (i,j),$ 
where $p_{k,l}: \GL(n,\F) \to \F, \ p_{k,l}(X)=x_{k,l}$ is the canonical coordinate projection. 

Denote by $\H(n,\F)$  the $2n+1$-dimensional Heisenberg group over a field $\F$.
 More precisely, define $\H(n,\F)$ as the following subgroup of $\UT(n+2,\F)$ 
$$ \H(n,\F):=
\Bigg\{\left( \begin{array}{ccc}
1 & a & b \\
0 & I_n & c \\
0 & 0 & 1 
\end{array}\right)\bigg |  \ a,c\in \F^n, b\in \F \Bigg\}, 
$$
where $I_n$ is the identity matrix of size $n$.
As a corollary of \cite[Proposition 2.9]{MEG95} we have the following.
\begin{corollary}\label{cor:rmH}
	Let 
 $G$ be a topological subgroup of $\GL(n+2,\F)$ containing $\H(n,\F)$.
	If the corner $1$-parameter subgroup $G_{1,n+2}$ of $\H(n,\F)$ is relatively minimal in $G$, then $\H(n,\F)$ is relatively minimal in $G.$
\end{corollary}
The proof of the following proposition heavily relies on the algebraic structure of the  matrix groups involved.
\begin{proposition}\label{prop:rmncomp} 
	Let $\F$ be a  non-discrete locally retrobounded complete field. 
	Then the subgroup $\mathsf{N}=\UT(n,\F)$ is relatively minimal
	in $\STP(n,\F).$	
\end{proposition}
\begin{proof}
	By Theorem \ref{t1}, $\STP(2,\F)$ is minimal. In particular,
	its subgroup 
	 $\UT(2,\F)$ is relatively minimal
	in $\STP(2,\F).$	The corner 1-parameter group $G_{1,3}$ is a subgroup of  $$ P:= 
	\Bigg\{\left(\begin{array}{ccc}
	a & 0 &  b \\
	0 & 1 & 0 \\
	0 & 0 & a^{-1} 
	\end{array}\right)\bigg |  \ a,b \in \F\Bigg\}.
	$$ 	
	Observe that $P$ 
	is topologically  isomorphic to the minimal group $\STP(2,\F)$. So, $G_{1,3}$ is relatively minimal in $P$ and hence also in the larger group $\STP(3,\F)$. By Corollary \ref{cor:rmH}, the Heisenberg group $\UT(3,\F)=\H(1,\F)$ is relatively minimal in $\STP(3,\F).$
	
	Continuing by induction on $n$ and assuming that $\UT(n,\F)$ is relatively minimal
	in $\STP(n,\F)$, we will prove that  $\UT(n+2,\F)$ is relatively minimal
	in $\STP(n+2,\F).$ Fix $n\geq 2$ and observe that $\H(n,\F)$ is a normal subgroup of $\STP(n+2,\F).$ In particular, $\H(n,\F)$ is a normal subgroup of $\UT(n+2,\F).$

	Moreover, we have   \begin{equation}\label{eq:wtilde}\UT(n+2,\F)=\widetilde{\UT(n,\F)}\H(n,\F),\end{equation}
	where 
	$$\widetilde{\UT(n,\F)}= \Bigg\{\left( \begin{array}{ccc}
	1 & 0_{1\times n} & 0 \\
	0_{n\times 1} & X & 0_{n\times 1} \\
	0 & 0_{1\times n}  & 1 
	\end{array}\right)\bigg| \  X\in \UT(n,\F) \Bigg \}.$$

Indeed, if  $X\in \UT(n,\F),\ a,c\in \F^n$ and  $b\in \F$, then	
$$	\left(\begin{array}{ccc}
		1 & 0_{1\times n} & 0 \\
		0_{n\times 1} & X^{-1} & 0_{n\times 1} \\
		0 & 0_{1\times n}  & 1 
	\end{array}\right) \cdot
\left(	\begin{array}{ccc}
	1 & a & b \\
	0_{n\times 1} & X & c \\
	0 & 0_{1\times n}  & 1 
\end{array}\right) \in  \H(n,\F).$$

	\noindent
	
	\begin{claim} \label{cl:1}  
	$\widetilde{\UT(n,\F)}$ and $\H(n,\F)$ are relatively minimal in $\STP(n+2,\F).$ 
	\end{claim} 
	\begin{proof}
		Denote by $\psi \colon \widetilde{\STP(n,\F)} \to \STP(n,\F)$ the natural topological group isomorphism 
		from  $$\widetilde{\STP(n,\F)}= \Bigg\{\left( \begin{array}{ccc}
		1 & 0_{1\times n} & 0 \\
		0_{n\times 1} & X & 0_{n\times 1} \\
		0 & 0_{1\times n}  & 1 
		\end{array}\right) \bigg | \ X\in \STP(n,\F) \Bigg \}$$ onto $\STP(n,\F).$ Since $\psi(\widetilde{\UT(n,\F)})=\UT(n,\F),$ we deduce by the induction hypothesis that
		$\widetilde{\UT(n,\F)}$ is relatively minimal in $\widetilde{\STP(n,\F)}$ and hence also in the larger group  $\STP(n+2,\F).$

		The corner 1-parameter group $G_{1,n+2}$ is a subgroup of  $$ P:= 
		\Bigg\{\left(\begin{array}{ccc}
		a & 0_{1\times n} &  b \\
		0_{n\times 1} & I_{n} & 0_{n\times 1} \\
		0 & 0_{1\times n} & a^{-1} 
		\end{array}\right) \bigg | \ a\in \F^{\times}, b \in \F\Bigg\}
		$$ and $P$ is topologically  isomorphic  (Theorem \ref{t1}) to the minimal  group $\STP(2,\F).$   So, $G_{1,n+2}$ is relatively minimal in $P$ and also in the larger group $\STP(n+2,\F)$. Now by Corollary \ref{cor:rmH}, the Heisenberg group $\H(n,\F)$ is relatively minimal in $\STP(n+2,\F).$
	\end{proof}

\vskip 0.2cm 
	Let $\sigma\subseteq \tau_p$ be a coarser Hausdorff group topology on $\STP(n+2,\F),$ where $\tau_p$ is the given (pointwise) topology.
	Clearly, $\H(n,\F)\cap \widetilde{\UT(n,\F)}$ is trivial. So by (\ref{eq:wtilde}), we deduce that $\UT(n+2,\F)$ is algebraically isomorphic to $\H(n,\F)\rtimes \widetilde{\UT(n,\F)}$. 
	
	\noindent
	
	\vskip 0.3cm 
	\begin{claim} \label{cl:2}  
	 $(\UT(n+2,\F), \sigma|_{\UT(n+2,\F)})$ is topologically isomorphic to  $$(\H(n,\F),\sigma|_{\H(n,\F)})\rtimes (\widetilde{\UT(n,\F)},\sigma|_{\widetilde{\UT(n,\F)}}).$$
	\end{claim} 
 
	\begin{proof}
		Consider the quotient map 
		$$q \colon (\STP(n+2,\F),\sigma)\to \STP(n+2,\F)/\H(n,\F).$$ From Claim \ref{cl:1} we obtain that  $\sigma|_{\H(n,\F)}=\tau_p|_{\H(n,\F)}$. So, the completeness of $\F$ implies that $\H(n,\F)$ is $\sigma$-closed in $\STP(n+2,\F)$. This means that $\sigma/\H(n,\F)$ is Hausdorff.
		Clearly, $\widetilde{\STP(n,\F)}\cap \H(n,\F)$ is trivial. Hence, the restriction  $$q|_{\widetilde{\STP(n,\F)}}\colon (\widetilde{\STP(n,\F)}, \sigma|_{\widetilde{\STP(n,\F)}}) \to q(\widetilde{\STP(n,\F))}$$ is a continuous isomorphism into a Hausdorff group. 
By the induction hypothesis, 	$\widetilde{\UT(n,\F)}$ is relatively minimal in $\widetilde{\STP(n,\F)}$ and  $$q|_{\widetilde{\UT(n,\F)}}\colon  (\widetilde{\UT(n,\F)},\sigma|_{\widetilde{\UT(n,\F)}})\to q({\widetilde{\UT(n,\F)}})$$ is a topological group isomorphism. 
		 Since 
		$q({\widetilde{\UT(n,\F)}})=\UT(n+2,\F)/\H(n,\F)$ and using \cite[Proposition 6.17]{RoD}, we deduce that $(\UT(n+2,\F), \sigma|_{\UT(n+2,\F)})$ is topologically isomorphic to the semidirect product $$(\H(n,\F),\sigma|_{H(n,\F)})\rtimes (\widetilde{\UT(n,\F)},\sigma|_{\widetilde{\UT(n,\F)}}).$$
	\end{proof}

By Claim \ref{cl:1}, $(\sigma|_{\UT(n+2,\F)})|_{\H(n,\F)}=(\tau_p|_{\UT(n+2,\F)})|_{\H(n,\F)}.$
By Claim \ref{cl:2}, 
$$(\sigma|_{\UT(n+2,\F)})/{\H(n,\F)}=\sigma|_{\widetilde{\UT(n,\F)}}$$ and 
$$(\tau_p|_{\UT(n+2,\F)})/{\H(n,\F)}=\tau_p|_{\widetilde{\UT(n,\F)}}.$$
Using Claim \ref{cl:1} again, we obtain that $\sigma|_{\widetilde{\UT(n,\F)}}=\tau_p|_{\widetilde{\UT(n,\F)}}$. It follows that $$(\sigma|_{\UT(n+2,\F)})/{\H(n,\F)}=(\tau_p|_{\UT(n+2,\F)})/{\H(n,\F)}$$ and by Merson's Lemma (Fact \ref{firMer}) we deduce that $\sigma|_{\UT(n+2,\F)}=\tau_p|_{\UT(n+2,\F)},$ as needed.  
\end{proof}


\begin{lemma}\label{lemma:comp}
	For every $1\leq i\leq n-1,$ let $E_{i,i+1}\in G_{i,i+1}$ with $p_{i,i+1}(E_{i,i+1})=1.$
	Then for every
	$B\in 	\mathsf{A},$ we have 
	\begin{equation}\label{eq:fir2}p_{i,i+1}(\alpha(B,E_{i,i+1}))= p_{i,i+1}(BE_{i,i+1}B^{-1})= p_{i,i}(B)(p_{i+1,i+1}(B))^{-1}.
	\end{equation} 
\end{lemma}
\begin{proof}
	Easy calculations.  
\end{proof}

 \vskip 0.2cm 
\begin{lemma}\label{lem:cmaps}
	Let $\F$ be a topological field and $n\geq 2$ be a positive number.  Suppose that $\tau$ is a group topology on $\mathsf{A}$   such that
	all $n-1$ actions $$\alpha_i\colon (\mathsf{A},\tau)\times (G_{i,i+1},\tau_p)\to (G_{i,i+1},\tau_p), \ \ \ i \in \{1, \cdots, n-1\}$$ are continuous, where $\tau_p$ is the pointwise topology and $\alpha_i=\alpha|_{\mathsf{A}\times G_{i,i+1}}.$ Then \ben
	\item 
	the homomorphism $$t_i\colon  \mathsf{A}\to \F^{\times}, \ t_i(B)=(p_{1,1}(B))(p_{i+1,i+1}(B))^{-1}$$ is continuous for every $1\leq i\leq n-1$;
	\item the homomorphism $m_i \colon \mathsf{A}\to \F^{\times}, \ m_i(B)=(p_{i,i}(B))^n$ is continuous for every $1\leq i\leq n.$
	\een
\end{lemma}
\begin{proof}
	(1) Since $\alpha_1 \colon (\mathsf{A},\tau)\times (G_{1,2},\tau_p)\to (G_{1,2},\tau_p)$ is continuous and $\tau_p$ is the pointwise topology, (\ref{eq:fir2}) 
	guarantees that $t_1$ is continuous. Now assume that $t_{i-1}$ is continuous and let us see that $t_i$ is continuous. 
Using (\ref{eq:fir2}) again, in view of the continuity of $\alpha_i,$ we deduce that the homomorphism $$\psi\colon  \mathsf{A}\to \F^{\times}, \ \psi(B)=p_{i,i}(B)(p_{i+1,i+1}(B))^{-1}$$ is continuous. The equality
	$t_{i}(B)=t_{i-1}(B)\psi(B)$ completes the proof.\\
	(2) For every $B\in \mathsf{A}$ we have $\prod_{i=1}^{n}p_{i,i}(B)=1$. This implies that $\prod_{i=1}^{n-1}t_{i}=(p_{1,1})^n$. By item $(1)$ and the fact that $\F$ is a topological field, we deduce that $m_1=(p_{1,1})^n$ is continuous.  We use the equality $m_i=m_1(t_{i-1})^{-n}$ to establish the continuity of $m_i$ for every $1<i\leq n.$
\end{proof}

\vskip 0.3cm 
\subsection{The action $\widetilde{\alpha}$ }\label{sbs:act}
Denote by $\tau_p$ the original pointwise topology on $\STP(n,\F)$ and by $\widetilde{\tau_p}$ the quotient topology on $\STP(n,\F)/Z$ with respect to the homomorphism  
$$q\colon \STP(n,\F)\to \STP(n,\F)/Z,$$  where $Z=Z(\SL(n,\F)).$
 The continuous action  $\alpha\colon (\mathsf{A},\tau_{p}|_\mathsf{A})\times (\mathsf{N},\tau_{p}|_\mathsf{N})\to (\mathsf{N}, \tau_p|_\mathsf{N})$ induces the action 
 $$\widetilde{\alpha}\colon (q(\mathsf{A}),\widetilde{\tau_p}|_{q(\mathsf{A})})\times (q(\mathsf{N}),\widetilde{\tau_p}|_{q(\mathsf{N})})\to (q(\mathsf{N}),\widetilde{\tau}|_{q(\mathsf{N})}).$$ 

Taking into account that $\STP(n,\F)\cong \mathsf{N}\rtimes_{\alpha} \mathsf{A}$ and the intersection $q(\mathsf{A})\cap q(\mathsf{N})$ is trivial,
one may identify $\STP(n,\F)/Z$ with the topological semidirect product $q(\mathsf{N})\rtimes_{\widetilde\alpha}~q(\mathsf{A}).$ 

 \vskip 0.1cm
  The next lemma will be used to prove the continuity of $\widetilde{\alpha}$.
 
\begin{lemma}\label{lem:opmap}
	The map $q|_\mathsf{N}\colon \mathsf{N}\to q(\mathsf{N})$ is a topological isomorphism.
\end{lemma}
\begin{proof}
The homomorphism $q|_\mathsf{N}$ is a bijection because $\mathsf{N} \cap Z$ is trivial. 
It suffices to show that $q|_\mathsf{N}\colon \mathsf{N}\to q(\mathsf{N})$ is an open map. Observe that 
$q$ is open and $q^{-1}(q(\mathsf{N}))=\mathsf{N}Z$. This implies that the restriction map 
$$q|_{\mathsf{N}Z} \colon  q^{-1}(q(\mathsf{N}))=\mathsf{N}Z\to q(\mathsf{N})$$ is also open. Having finite index in $\mathsf{N}Z$, the closed subgroup $\mathsf{N}$ is open in $\mathsf{N}Z$. It follows that $q|_\mathsf{N}$ is an open map.
\end{proof}

\vskip 0.1cm 
\begin{lemma}\label{lem:atilcon}
	The action $\widetilde{\alpha} \colon (q(\mathsf{A}),\widetilde{\tau_p}|_{q(\mathsf{A})})\times (q(\mathsf{N}),\widetilde{\tau_p}|_{q(\mathsf{N})})\to (q(\mathsf{N}),\widetilde{\tau_p}|_{q(\mathsf{N})})$ is continuous.
\end{lemma}
\begin{proof} We have the following commutative diagram
	\begin{equation} \label{diag1}
	\xymatrix{ \mathsf{A} \ar@<-2ex>[d]_{q} \times \mathsf{N} \ar@<2ex>[d]^{q} \ar[r]^{\alpha}  & \mathsf{N} \ar[d]^{q} \\
		q(\mathsf{A}) \times q(\mathsf{N}) \ar[r]^{\tilde{\alpha}}  &  q(\mathsf{N}) }
	\end{equation} 
	Fix an arbitrary $(a,n)\in \mathsf{A}\times \mathsf{N}$ and let $U$ be a $\widetilde{\tau_p}|_{q(\mathsf{N})}$-neighborhood of \[\widetilde{\alpha}(q(a),(q(n))=q(\alpha(a,n)).\] By the continuity of $q|_\mathsf{N}$, there exists  a $\tau_p|_\mathsf{N}$-neighborhood $V$  of $\alpha(a,n)$ such that $q(V)\subseteq U$. By the continuity of $\alpha$, there exist 
	a $\tau_p|_\mathsf{A}$-neighborhood $W$ of $a$ and a $\tau_p|_\mathsf{N}$-neighborhood of $n$ such that $\alpha(W\times O)\subseteq V$. 
Since $q^{-1}(q(\mathsf{A}))=\mathsf{A},$ it follows that $q|_\mathsf{A}$ is open.
	 By Lemma \ref{lem:opmap}, also $q|_\mathsf{N}$ is 
	 open. So, $q(W)$ is a $\widetilde\tau_p|_{q(\mathsf{A})}$-neighborhood $W$ of $q(a)$ and  $q(O)$ is a $\widetilde\tau_p|_{q(\mathsf{N})}$-neighborhood of $q(n)$. Then \[q(\alpha(W\times O))=\widetilde{\alpha}(q(W)\times q(\mathsf{A}))\subseteq q(V)\subseteq U\] which proves the continuity of $\widetilde{\alpha}$ in $(q(a),q(n))$.
\end{proof}


\begin{proposition}
		Let $\F$ be a  non-discrete locally retrobounded complete field. 
		Then	$q(\mathsf{N})$ is $q(\mathsf{A})$-minimal with respect to the action $\widetilde{\alpha}.$
\end{proposition}
\begin{proof}
	By Lemma \ref{lem:atilcon}, the action $\widetilde{\alpha}$ is $(\widetilde{\tau_p}|_{q(\mathsf{A})}, \widetilde{\tau_p}|_{q(\mathsf{N})}, \widetilde{\tau_p}|_{q(\mathsf{N})})$-continuous.
	Let $\sigma\subseteq \widetilde{\tau_p}|_{q(\mathsf{N})}$ be a coarser Hausdorff group topology such that $\widetilde{\alpha}$ is $(\widetilde{\tau_p}|_{q(\mathsf{A})}, \sigma, \sigma)$-continuous. We have to show that $\sigma=\widetilde{\tau_p}|_{q(\mathsf{N})}.$ 
	
	Let us see that $\alpha$ is $(\tau_p|_\mathsf{A}, (q|_\mathsf{N})^{-1}(\sigma), (q|_\mathsf{N})^{-1}(\sigma))$-continuous. Indeed, this follows from the equality \[q|_\mathsf{N}\circ\alpha=\widetilde\alpha\circ(q|_\mathsf{A}\times q|_\mathsf{N})\] and the 
	$(\widetilde{\tau_p}|_{q(\mathsf{A})}, \sigma, \sigma)$-continuity of  $\widetilde{\alpha}$. 
	Since $q|_\mathsf{N}$ is an injection  
	 and $\sigma$ is a Hausdorff group topology on $q(\mathsf{N})$, then clearly 
	$(q|_\mathsf{N})^{-1}(\sigma)\subseteq \tau_p|_\mathsf{N}$ is a coarser Hausdorff group topology on $\mathsf{N}$. By Proposition \ref{prop:rmncomp} and Fact  \ref{f:rel-min=G-min}, $\mathsf{N}$ is $\mathsf{A}$-minimal with respect to the action $\alpha.$ In particular, we deduce that 
	$(q|_\mathsf{N})^{-1}(\sigma)= \tau_p|_\mathsf{N}$. 
	This implies  that  $\sigma=\widetilde{\tau_p}|_{q(\mathsf{N})},$ 
	which completes the proof.
\end{proof}

Using Fact  \ref{f:rel-min=G-min}, 
 we immediately obtain:
\begin{corollary}\label{cor:rmqsut} 
Let $\F$ be a  non-discrete locally retrobounded complete field. 
Then the subgroup $q(\mathsf{N})$ is relatively minimal in $\STP(n,\F)/Z$.
\end{corollary}

\vskip 0.3cm 
\subsection{When $\F$ is a local field}

  It is easy to see that for every $1\leq i\leq n-1,$ there exists a continuous \textit{central retraction} $r$ from 
  $q(\mathsf{N})$ to its $q(\mathsf{A})$-subgroup  $q(G_{i,i+1}).$ 
This means that $r(q(x)q(a)q(x)^{-1})=q(a)$ for every $x \in \mathsf{N}$ and $a \in G_{i,i+1}$. 

\vskip 0.3cm
The following fact will be used to prove Theorem \ref{thm:qbycen} which provides sufficient conditions for the minimality of $\STP(n,\F)/Z.$

\begin{fact} \cite[Proposition 2.7]{MEG95} \label{fac:min95}
	Let $M=(X \rtimes_{\alpha} G, \gamma)$ be a topological semidirect
	product and $\{Y_i\}_{i \in I}$ be a system of $G$-subgroups in
	$X$ such that the system of actions
	$$
	\{\a|_{G \times Y_i} \colon  G \times Y_i \to Y_i\}_{i \in I}\
	$$
	is \textit{t-exact} 
	(that is, there is no strictly coarser 
	(not necessarily Hausdorff) group topology on $G$ such that all actions remain continuous). 
	Suppose that for each $i \in I$ there exists a
	continuous central retraction $q_i \colon  X \to Y_i$. Then if $\gamma_1
	\subseteq \gamma$ is a coarser group topology on $M$ such that $\gamma_1|_X
	=\gamma|_X,$ then $\gamma_1=\gamma$.
\end{fact}

The proof of the next proposition was inspired by the proof of  the total minimality of $\SL(2,\R)$ given in \cite[Theorem 7.4.1]{DPS89}.
\begin{proposition}	\label{pro:locf}
Let $\F$ be a local field and $n\geq 2$. 
Then the system of $n-1$ actions 
 $$\{\widetilde{\alpha_i}:(q(\mathsf{A}),\widetilde{\tau_p})\times (q(G_{i,i+1}),\widetilde{\tau_p})\to (q(G_{i,i+1}),\widetilde{\tau_p})| \ \ \ i \in \{1, \cdots, n-1\}\}$$ is $t$-exact.
\end{proposition}
\begin{proof}
Recall that $\F$ admits an absolute value $|\cdot |$.  Let $\sigma\subseteq \widetilde \tau_p$ be a coarser  group topology on $q(\mathsf{A})$  such that all $n-1$ actions $$\widetilde{\alpha_i} \colon (q(\mathsf{A}),\sigma)\times (q(G_{i,i+1}),\widetilde{\tau_p})\to (q(G_{i,i+1}),\widetilde{\tau_p})$$  are continuous. This implies that the $n-1$ actions $$\alpha_i \colon (\mathsf{A},q^{-1}(\sigma))\times (G_{i,i+1},\tau_p)\to (G_{i,i+1},\tau_p)$$  are continuous. By Lemma \ref{lem:cmaps}(2), the homomorphism
$$m_i=(p_{i,i})^n \colon  (\mathsf{A},q^{-1}(\sigma))\to \F^{\times}$$
 is continuous. If $q(B)=q(C)$ then $B^{-1}C=\lambda I$, where $\lambda^n=1.$ It follows  
that the map $\widehat{m_i} \colon (q(\mathsf{A}),\sigma)\to  \F^{\times}$ defined by $\widehat{m_i}\circ q=m_i$ is well-defined and continuous for every $1\leq i\leq n$.   Consider an arbitrary net 
$\{\eps_\alpha\}_\alpha$ in
$\mathsf{A}$, 
 such that $\lim q(\eps_\alpha)=q(I)$ in $\sigma$. By the continuity of $\widehat{m_i}$, we deduce that $\lim(p_{i,i})^n(\eps_\alpha)=1$. In particular,   the nets $\{p_{i,i}(\eps_\alpha)\}$, where $1\leq i\leq n$, are bounded with respect to the absolute value.   Hence, there exists a $\sigma$-neighborhood $V$ of $q(I)$ that is contained in a compact subset of $q(\mathsf{A})$. This implies that $\sigma=\widetilde \tau_p$. 
\end{proof}

\vskip 0.1cm
\begin{theorem}\label{thm:qbycen}
Let $\F$ be a local field. 
Then $\STP(n,\F)/Z$ is minimal for every $n\in\N$.
\end{theorem} 
\begin{proof}
Clearly, we may assume that $n\geq 2$. By Corollary \ref{cor:rmqsut}, the subgroup 	$q(N)$ is relatively minimal in $q(\STP(n,\F))= \STP(n,\F)/Z$.
By Proposition \ref{pro:locf}, the system of $n-1$ actions $$\{\widetilde{\alpha_i} \colon (q(\mathsf{A}),\widetilde{\tau_p})\times (q(G_{i,i+1}),\widetilde{\tau_p})\to (q(G_{i,i+1}),\widetilde{\tau_p})\}$$
 is $t$-exact. Using Fact \ref{fac:min95} we complete the proof. 
\end{proof}   

 In case $G/L$ is sup-complete for every closed normal subgroup $L$ of $G,$ then $G$ is called {\it totally sup-complete}. In particular, if $G$ is  either a compact group or a sup-complete  (topologically) simple group, then
it is totally sup-complete.

\begin{fact} \cite[Theorem 7.3.1]{DPS89}\label{fac:cns}
	Let $G$ be a topological group and let $L$ be a closed  normal subgroup  of $G$ which is (totally) sup-complete. If $L$ and $G/L$ are both (totally) minimal, then $G$ is (totally) minimal, too.
\end{fact}

Since $Z=Z(\SL(n,\F))$ is finite  and using Fact \ref{fac:cns} and Theorem \ref{thm:qbycen}, we 
obtain 
one of our main results:
\begin{theorem} \label{thm:sutmin}
	Let $\F$ be a local field. 
	Then $\STP(n,\F)$ is minimal for every $n\in\N$.
\end{theorem} 
One can consider
the topological group $\TM(n,\F)$  of lower triangular $n\times n$ matrices over $\F$ and its subgroup  $\STM(n,\F)= \TM(n,\F)\cap \SL(n,\F).$ It is easy to see that $\STP(n,\F)$ is topologically isomorphic to $\STM(n,\F)$. So, Theorem \ref{thm:sutmin}  immediately implies:

\begin{corollary}
	Let $\F$ be a local field.
Then $\STM(n,\F)$ is minimal for every $n\in\N$.
\end{corollary}

\vskip 0.5cm   
\section{Minimality properties of $\SL(n,\F)$ and $\PGL(n,\F)$} \label{s:4}

It is known that an archimedean local field is either the field of reals $\R$ or the field of complex numbers $\C.$

The following Iwasawa decomposition of $\SL(n,\F)$ (see \cite{AHL16,B,PR,Tits}) plays a key role in proving Theorem \ref{t:SLSMALL}.

\begin{fact}\label{fac:iwasawa}
	Let $\F$ be a local field. Then there exists a compact subgroup $K$ of $\SL(n,\F)$ such that  $\SL(n,\F)= \STP(n,\F)K$. In particular, \ben
	\item if $\F=\R$, then $K$ is the orthogonal group $\O(n,\R)$;
	\item if $\F=\C$, then $K$ is the special unitary group  $\SU(n,\C)$;
	\item if $\F$ is non-archimedean, then $K=\SL(n,\mathcal O_{\F})$, where $\mathcal O_{\F}$ is the ring of integers of $\F$, 
	namely, $\mathcal O_{\F}=\{a\in \F: |a|\leq 1\}$.
	\een
\end{fact}


Recall that a subgroup $H$ of a topological group $G$ is said to be \emph{co-compact} if the coset space $G / H$ is compact. 
If $G=KH$ (equivalently, $G=HK$) for some compact subset $K$ of $G$ and a subgroup $H$, then 
 $H$ is co-compact in $G$. Indeed, let $q: G \to G/H, x \mapsto xH$ be the natural projection. Then its restriction on $K$ is onto because $G=KH$. So, $q(K)=G/H$ is also compact.  Since $Z$ is finite, we obtain the following as a corollary of Fact \ref{fac:iwasawa}:
\begin{corollary}\label{cor:cocomp}
	Let $\F$ be a local field. Then $\STP(n,\F)/Z$ is co-compact in $\PSL(n,\F)$.
\end{corollary}

A subgroup $H$ of a Hausdorff topological group $(G, \tau)$ is called \textit{strongly closed}, \cite{DM10} if $H$ is
$\sigma$-closed for every Hausdorff group topology $\sigma \subseteq \tau$ on $G$.

	\vskip 0.2cm
\begin{theorem} \label{pro:pslntm} Let $\F$ be a local 
 	 field. 
 	  Then $\SL(n,\F)$ and the projective special linear group
 	   $\PSL(n,\F)=\SL(n,\F)/Z(\SL(n,\F))$ are totally minimal for every $n\in \N.$ 
\end{theorem}
\begin{proof}
We may assume that $n\geq 2.$	By 
Fact \ref{fac:simple}, 
 $\PSL(n,\F)$ is simple 
so it suffices to prove that $\PSL(n,\F)$ is minimal.
 By Theorem \ref{thm:qbycen}, $G :=\STP(2,\F)/Z(\SL(n,\F))$ is minimal. So, in particular, $G$ is relatively minimal in $\PSL(n,\F)$. Furthermore, $G$ is also 
	sup-complete 
since $G$ is locally compact. 
	So we obtain that $G$ is strongly closed. Then the subgroup $G$ is also co-minimal in $\PSL(n,\F)$, being co-compact by  Iwasawa decomposition. 	It follows from Fact \ref{MER} that  $\PSL(n,\F)$ is minimal.	 By Fact \ref{fac:cns}, $\SL(n,\F)$ is  totally minimal.
\end{proof} 

\begin{remark} \label{rem:BG} Bader and Gelander (\cite[Corollary 5.3]{BG}) recently proved that every separable  quasi-semisimple group is totally minimal. 
	Then by \cite[Ch. I, Proposition (1.2.1)]{MA} every Zariski-connected 
		semi-simple group (e.g., $\SL(n,\F)$) over a local field $\F$ is quasi-semisimple. 
		It follows that for every 
		local field $\F$ (so also when $\mathrm{char}(\F)=2$) the 
		groups $\SL(n,\F)$ and 	$\PSL(n,\F)$ are totally minimal.
\end{remark}

The following concept has a key role in the Total Minimality Criterion. 
\begin{definition}
A subgroup $H$ of a topological group $G$ is totally dense if for every closed normal subgroup
$L$ of $G$ the intersection $L \cap H$ is dense in $L.$	
\end{definition} 

\begin{fact}\label{fact:tmc} \cite[Total Minimality Criterion]{DP}
 Let $H$ be a dense subgroup of a topological group $G$. Then $H$ is
	totally minimal if and only if $G$ is totally minimal and $H$ is totally dense in $G$.
\end{fact}

In the sequel we no longer assume that $\mathrm{char}(\F)\neq 2$ in view of Remark \ref{rem:BG}.
  
  \vskip 0.2cm 
\begin{theorem}\label{t:SLSMALL} 
	Let $\F$ be a 
	subfield of a local field.
 Then the following conditions are equivalent:\ben \item $\PSL(n,\F)$ is totally minimal; \item   $\SL(n,\F)$ is totally minimal; \item
	 $Z(\SL(n,\F))=Z(\SL(n,\widehat{\F}))$ 
	 (i.e., $\mu_n(\F)=\mu_n(\widehat{\F})$).  
	 \een 
\end{theorem}
\begin{proof} If $\F$ is finite, then $\F=\widehat{\F}$ and conditions $(1), (2),(3)$ are all satisfied. So we may assume that $\F$ is infinite and $\widehat{\F}$
	is a local field in view of Remark \ref{r:standing}(2).
	
$(1)\Rightarrow (2)$: Use Fact \ref{fac:cns} with  $G=\SL(n,\F)$ and its finite subgroup $L=Z(\SL(n,\F)).$\\
$(2)\Rightarrow (3)$:  Let $G:=\SL(n,\widehat{\F}), \ H:=\SL(n,\F)$ and  suppose that $Z(\SL(n,\F)) \neq Z(\SL(n,\widehat{\F})).$ Then $L:=Z(\SL(n,\widehat{\F}))$ is a closed normal subgroup of $G,$ and $L\cap H=Z(\SL(n,\F))$ is not dense in $L$, being a finite proper subgroup of $L$.
So $\SL(n,\F)$ is not totally dense in $\SL(n,\widehat{\F})$. By the Total Minimality Criterion, we deduce that  $\SL(n,\F)$ is not totally minimal.\\
$(3)\Rightarrow (1)$: By Theorem \ref{pro:pslntm}, the group $\PSL(n,\widehat\F)$ is totally minimal. 
Since $Z(\SL(n,\F))$ $=Z(\SL(n,\widehat{\F})),$ we deduce that $\PSL(n,\F)$ is dense in $\PSL(n,\widehat\F)$.   As $\PSL(n,\widehat\F)$ is simple 
(Fact \ref{fac:simple}), 
its dense subgroup $\PSL(n,\F)$ is,  in fact, totally dense. By the Total Minimality Criterion,  $\PSL(n,\F)$ is also totally minimal.
\end{proof}

	\vskip 0.2cm
\begin{corollary}\label{cor:tmsln}
	Let $\F$ be a local field. 
	If $Z(\SL(n,\F))\subseteq \{I,-I\},$ then for every topological subfield $H$ of $\F$ the groups
	$\PSL(n,H)$ and $\SL(n,H)$ are totally minimal. In particular,\ben  \item $\SL(2,H)$ and $\PSL(2,H)$  are totally minimal for every topological subfield $H$ of $\F;$ 
 \item $\SL(n,H)$ and $\PSL(n,H)$  are totally minimal for every topological subfield $H$ of $\R.$\een 
\end{corollary}

\vskip 0.3cm 
\subsection{Total minimality of $\PGL(n,\F)$}
The next result is probably known. 
We prove it for the sake of 
completeness. Perhaps it can be derived also from the results of \cite[Section 26]{Lorenz}.
\begin{lemma} \label{lem: fstcom}
	Let $\F$ be a local field, 
	$n \in \N,$ and  $M_n=\{x^n| \ x\in\F^{\times}\} $.
	Then 
	\ben \item $M_n$ is closed in $\F^{\times}$; \item the group $\F^{\times}/M_n$ is compact.\een
\end{lemma}
\begin{proof}
	Recall that the local field $\F$ admits an absolute value $|\cdot |$.\\
	$(1)$ Suppose that $\lim_{m\to \infty} (x_m)^n=y\in \F^{\times}$, where $\{x_m\}_{m\in \N}$ is a sequence contained in $\F^{\times}$. We have to show that
	$y\in M_n.$  Clearly, the sequence  $\{x_m\}_{m\in \N}$ is bounded with respect to the absolute value. Since $\F$ is a local field, there exists a 
	subsequence $\{x_{m_l}\}_{l\in \N}$ of $\{x_m\}_{m\in \N}$ such that $\lim_{l\to \infty} x_{m_l}=t$ for some $t\in \F.$ Since $\F$ is a Hausdorff topological field, it follows that $\lim_{l\to \infty} (x_{m_l})^n=t^n=y$. Clearly, $t\neq 0$ and we deduce that $y\in M_n.$
	
	$(2)$ It is easy to see that if $\F\in \{\R, \C\}$, then 
	$\F^{\times}/M_n$ is  
	a group with at most two elements.
	So, we may assume that $\F$ is a non-archimedean local field. 	In this case, 
	the value group (i.e., the set $\{|x|: \ |x|\neq 0\}$)
	is the  infinite cyclic closed subgroup $\{a^k| \ k\in \Z \}$ of $\R^{\times}$, where $a := \max\{|x| :|x|<1\}$ (see \cite{V}). Let $r \colon  \F^{\times}\to \F^{\times}/
	M_n$ be  the quotient map. It suffices to show that for every sequence $\{x_m\}_{m\in \N}\subseteq \F^{\times}$ there exists a subsequence $\{x_{m_l}\}_{l\in \N}$ such that the sequence $\{r(x_{m_l})\}_{l\in \N}$ converges in 
	$\F^{\times}/M_n.$
	For every $m\in \N,$ we have $|x_m|=a^{t_m}$ for some $t_m\in \Z$. There exist $s_m\in \Z$ and 
	$r_m\in [1-n,n-1]\cap\Z$ 
	with $t_m=ns_m+r_m$.
	Choose $\lambda_m\in  \F^{\times}$ with $|\lambda_m|=a^{-s_m}$. Letting $y_m= x_m\lambda_m^n,$ we obtain a  sequence $\{y_m\}_{m\in \N}$ with $r(x_m)=r(y_m)$ and 
	$a^{n-1}\leq|y_m|\leq a^{1-n}$ 
	for every $m\in \N.$   As the sequence $\{y_m\}_{m\in \N}$ is bounded and $\F$ is a local field, there exists a converging subsequence $\{y_{m_l}\}_{l\in \N}$. Since 
	$a^{n-1}\leq|y_m|$,  the limit is in $\F^{\times}$.
	As $r(x_m)=r(y_m)$ for every $m\in \N$ and using the continuity of $r,$ we deduce that the sequence $\{r(x_{m_l})\}_{l\in \N}$ converges in $\F^{\times}/M_n.$
\end{proof}

\vskip 0.3cm 
Recall that the center $Z(\GL(n,\F))$ is $\{\lambda I: \lambda \in \F^{\times}\}$.  Below  we denote it by $Z$.  
In the sequel, $\widetilde{\PSL(n,\F)}=q(\SL(n,\F))= (\SL(n,\F)\cdot Z)/Z$ 
is a  normal subgroup of $\PGL(n,\F)$, where $$q \colon \GL(n,\F)\to \GL(n,\F)/Z=\PGL(n,\F)$$ is the quotient map. The map $q$ induces a continuous  
isomorphism	
$$\phi:\PSL(n,\F)\to \widetilde{\PSL(n,\F)}.$$
 It is worth noting that  the second isomorphism theorem, which implies that $\phi$ is an algebraic isomorphism, does not hold in general for topological groups (see \cite[p. 14]{It}).

\vskip 0.2cm 
\begin{proposition}\label{pro:tmcocomp}
	Let $\F$ be a local field. 
	Then $\widetilde{\PSL(n,\F)}$ is a totally minimal totally sup-complete group  
	and 
	the factor group $\PGL(n,\F)/\widetilde{\PSL(n,\F)}$ is compact  Hausdorff.
\end{proposition}
\begin{proof}
	By Theorem \ref{pro:pslntm},  
	the group $\PSL(n,\F)$ is  totally minimal. Being locally compact and simple (see Fact \ref{fac:simple}),
	$\PSL(n,\F)$ is also totally sup-complete.
	Consider the continuous isomorphism 
	$\phi \colon \PSL(n,\F)\to \widetilde{\PSL(n,\F)}.$ Since $\PSL(n,\F)$ is totally minimal it follows that $\phi$ is in fact a topological group isomorphism, so  $\widetilde{\PSL(n,\F)}$ is 
	a totally minimal totally sup-complete group.  	
	
	As $\SL(n,\F)\cdot Z=\det^{-1}(M_n)$, Lemma  \ref{lem: fstcom}(1) implies that the factor group $$\GL(n,\F)/(\SL(n,\F)\cdot Z)$$ is Hausdorff.  
	The continuous homomorphism $$\psi \colon \F^{\times} \to \GL(n,\F), \ 
	\psi(\lambda)= \left(\begin{array}{cccc} 
	\lambda & 0 & \ldots & 	0 \\
	0 & 1 & \ddots  & \vdots\\
	\vdots & \ddots & \ddots & 0 \\
	0 & \ldots & 0 & 1
	\end{array} \right) $$
	induces a continuous isomorphism $\tilde{\psi}$ from 
	$\F^{\times}/M_n$ onto $\GL(n,\F)/(\SL(n,\F)\cdot Z)$.
	
	By  Lemma \ref{lem: fstcom}(2), $\F^{\times}/M_n$  is compact. 
	Hence $\tilde{\psi}$ is a topological group isomorphism. This proves that $\GL(n,\F)/(\SL(n,\F)\cdot Z)$ is compact.
	By the \textit{third isomorphism theorem for topological groups}
	(which is easy to verify for all topological groups; see \cite[Theorem 1.5.18]{AT} and \cite[Proposition 3.6]{It}),
	we have 
	$$\PGL(n,\F)/\widetilde{\PSL(n,\F)}=(\GL(n,\F)/Z)/((\SL(n,\F)\cdot Z)/Z) \cong \GL(n,\F)/(\SL(n,\F)\cdot Z),$$ 
	which completes the proof.
\end{proof}

\vskip 0.2cm

By Proposition \ref{pro:tmcocomp} and Fact  \ref{fac:cns}, we immediately obtain

\begin{theorem}\label{thm:ptm}
	Let $\F$ be a local field. 
	Then $\PGL(n,\F)$ is totally minimal.
\end{theorem}

Very recently, U. Bader informed us that Theorem \ref{thm:ptm} follows also from \cite[Theorem 3.4]{BG}.



\begin{theorem} \label{pro:sofr} 
	Let $\F$ be a topological subfield of $\R$ and $n$ be an odd number.  Then $\PGL(n,\F)$ is totally minimal.
\end{theorem}
\begin{proof}
	By Corollary \ref{cor:tmsln}(2), $\PSL(n,\F)$ is totally minimal. It follows that $\widetilde{\PSL(n,\F)}$ is also totally minimal. To establish the total minimality of $\PGL(n,\F),$ it suffices to show, in view of Fact
	\ref{fact:tmc},  that $\widetilde{\PSL(n,\F)}$ is dense in $\PGL(n,\F).$ Let us see first that $M_n$ is dense in 
	$\F^{\times}.$ If $a\in \F^{\times}$, then $\sqrt[n] a\in \R$. As $\Q\subseteq \F,$ there exists a sequence  $\{x_m\}_{m\in \N}\subseteq \F^{\times}$ converging to
	$\sqrt[n] a$. So $\lim_{m\to \infty} (x_{m})^n=a$, which proves that $M_n$ is dense in 
	$\F^{\times}.$ The continuous homomorphism 
	$$\psi \colon \F^{\times} \to \GL(n,\F), \ 
	\psi(\lambda)= \left(\begin{array}{cccc} 
	\lambda & 0 & \ldots & 	0 \\
	0 & 1 & \ddots  & \vdots\\
	\vdots & \ddots & \ddots & 0 \\
	0 & \ldots & 0 & 1
	\end{array} \right) $$ 
	induces a continuous homomorphism $\hat{\psi}$ from 
	$\F^{\times}$ onto $\PGL(n,\F)$. Since $\hat{\psi}(M_n)=\widetilde{\PSL(n,\F)},$ we deduce that $\widetilde{\PSL(n,\F)}$ is dense in $\PGL(n,\F)$, as needed.
\end{proof}

\vskip 0.2cm 
\begin{question} \label{q:PGL-Q} 
	Let $\F$ be a topological subfield of $\R$ and $n$ be an even number.	Is $\PGL(n,\F)$ (totally) minimal? 
\end{question}

\vskip 0.5cm 
 \section{Fermat primes and minimality of special linear groups}

The next proposition may be viewed as a counterpart of Theorem  \ref{t:SLSMALL}.

\begin{proposition}\label{pro:minsl}
Let $\F$ be a 
	subfield of a local field. 
Then the following conditions are equivalent: \ben \item  $\SL(n,\F)$ is   minimal; \item any non-trivial central subgroup of $\SL(n,\widehat{\F})$  intersects $\SL(n,\F)$ non-trivially. \een
\end{proposition}
\begin{proof}
	If $\F$ is finite, then $\F=\widehat{\F}$ and conditions $(1), (2)$ are  satisfied. So we may assume that $\F$ is infinite and $\widehat{\F}$
	is a local field in view of Remark \ref{r:standing}(2).

$(1)\Rightarrow (2)$: Immediately follows from the Minimality Criterion, as any non-trivial central subgroup of $\SL(n,\widehat{\F})$  is normal and closed (being finite).\\
$(2)\Rightarrow (1)$: Let us see that $\SL(n,\F)$ is essential in the minimal group $\SL(n,\widehat{\F})$. To this aim, let $L$ be a closed non-trivial normal subgroup of $\SL(n,\widehat{\F})$. If the central subgroup $Z(\SL(n,\widehat{\F}))\cap
L$ is non-trivial, then by our assumption $L\cap \SL(n,\F)$ is non-trivial. If $L$ trivially intersects  $Z(\SL(n,\widehat{\F}))$, then $q(L)$ is a nontrivial normal subgroup of
$\PSL(n,\widehat{\F}),$ where $q \colon \SL(n,\widehat{\F})\to \PSL(n,\widehat{\F})$ is the quotient map. Using the simplicity of $\PSL(n,\widehat{\F})$ 
(Fact \ref{fac:simple}), 
we deduce that $q(L)=\PSL(n,\widehat{\F}).$
Choose $A\in \SL(n,\F)$ such that $A$ is not a root of $I$. Since $q(L)=\PSL(n,\widehat{\F})$, there exist $X\in L$ and $n$th-roots of unity $\lambda,\mu\in \widehat{\F}$
such that $\lambda X=\mu A$. Therefore, $X^n=A^n$ is a non-trivial element of $L\cap \SL(n,\F)$.  This proves that 
$\SL(n,\F)$ is  essential in $\SL(n,\widehat{\F})$. By the Minimality Criterion, $\SL(n,\F)$ is  minimal being a dense essential subgroup of the minimal group
$\SL(n,\widehat{\F})$.
\end{proof}

\vskip 0.2cm 
\begin{corollary}\label{cor:p2}
Let $\F$ be a subfield of a local field.  
Then $\SL(2^k,\F)$ is minimal for every $k\in \N.$ If $\mathrm{char}(\F)=2,$ then  $\SL(2^k,\F)$ is totally minimal. 
\end{corollary}
\begin{proof} 
If $\mathrm{char}(\F)=2$ and $\lambda^{2^k}	=1$, where $\lambda\in \widehat{\F},$ then $\lambda=1.$ It follows that $Z(\SL(2^k,\widehat{\F}))$ is trivial. By Theorem \ref{t:SLSMALL},  $\SL(2^k,\F)$ is totally minimal. 

Now assume that $\mathrm{char}(\F)\neq 2$
and let $L$ be a non-trivial central subgroup of  $\SL(2^k,\widehat{\F})$. Then  $$I\neq -I\in \SL(2^k,\F)\cap
L.$$ This proves the minimality of $\SL(2^k,\F)$, in view of Proposition \ref{pro:minsl}.
\end{proof}

Since $\C$ is a local field, $\SL(n,\C)$ is totally minimal. We have the following trichotomy for $\Q(i)$.

\vskip 0.2cm 
\begin{corollary}\label{ex:slnotmin}
Let $n$ be a natural number.\ben \item If $n\in \{1,2,4\}$, then $\SL(n,\Q(i))$ is totally minimal.
\item  If $n=2^k$ for $k>2$, then $\SL(n,\Q(i))$ is minimal but not totally minimal.
\item If $n$ is not a power of $2$, then  $\SL(n,\Q(i))$ is not minimal.\een
\end{corollary}
\begin{proof}
$(1)$ If $n\in \{1,2,4\}$, then $Z(\SL(n,\Q(i)))= Z(\SL(n,\C))\subseteq \{\pm I, \pm iI\}.$  By Theorem \ref{t:SLSMALL}, $\SL(n,\Q(i))$ is totally minimal.

$(2)$ By Corollary \ref{cor:p2}, $\SL(n,\Q(i))$ is minimal. Let $\rho_8=e^{\frac {\pi i}{4}}$ be the $8$-th primitive root of unity. As $n=2^k$ for $k>2$, we have $\rho_8I \in Z(\SL(n,\C))$ but $\rho_8I \notin  Z(\SL(n,\Q(i))).$ So, $\SL(n,\Q(i))$ is not totally minimal by Theorem \ref{t:SLSMALL}.

$(3)$ 
One can show using the same arguments from Example \ref{ex:nmin} that the finite central subgroup  $L=\{\lambda I:\lambda^p=1\}$ of  $\SL(n,\C)$ trivially intersects  $\SL(n,\Q(i)).$ This means that $\SL(n,\Q(i))$ is not essential in $\SL(n,\C).$ By the Minimality Criterion, $\SL(n,\Q(i))$ is not minimal. 
\end{proof}

Now we consider the field of $p$-adic numbers $\Q_p$. It is known 
 that $\Q_p$ contains $p-1$ roots of unity in case $p>2$ and that $\pm 1$ are the only roots of unity in $\Q_2$ (see \cite[p. 15]{V}). 

\vskip 0.2cm 
\begin{corollary} \label{c:p-adic} 
	 Let $\F$ be a topological subfield of $\Q_p$.
\ben 
\item $\SL(n,\F)$ and $\PSL(n,\F)$ are totally minimal for every $n$ which is coprime to $p-1$ (e.g., arbitrary $n$ for $p=2$). 

\item If  
$p-1$ is not a power of $2,$
then $\SL(p-1,\Q)$ is not minimal. 
\een
\end{corollary}
\begin{proof}
	(1) Use  
	Corollary \ref{cor:tmsln}.

	(2) $\SL(p-1,\Q)$ is not essential in $\SL(p-1,\Q_p)$. Indeed, let $q$ be an odd prime dividing $p-1$. Then the finite central subgroup $L=\{\lambda I: \ \lambda^q=1\}$ of $\SL(p-1,\Q_p)$ trivially intersects $\SL(p-1,\Q)$. 	
\end{proof}

\vskip 0.2cm 
The next theorem 
is one of our main results. It 
 provides a characterization of Fermat primes via the minimality of some topological matrix groups. 
\begin{theorem}\label{thm:fermat}
 For an odd prime  $p$ the following conditions are equivalent: \ben\item $p$ is a  Fermat prime;
\item $\SL(p-1,(\Q,\tau_p))$ is minimal, where $(\Q,\tau_p)$ is the field of rationals with the $p$-adic topology;
\item  $\SL(p-1,\Q(i))$ is minimal, where $\Q(i) \subset \C$ is the Gaussian rational field.
\een  
\end{theorem}
\begin{proof}
If $p$ is a Fermat prime, then $p-1=2^k$ for some positive integer $k.$ By Corollary \ref{cor:p2}, $\SL(p-1,(\Q,\tau_p))$ and $\SL(p-1,\Q(i))$ are both minimal. If $p$ is not a Fermat prime, then $p-1$ is not a power of $2.$  By  Corollary \ref{c:p-adic}(2),  $\SL(p-1,(\Q,\tau_p))$ is not minimal.
By Corollary \ref{ex:slnotmin}(3), $\SL(p-1,\Q(i))$  is not minimal.
\end{proof}

\vskip 0.2cm 
\begin{remark} \ 
	\ben
\item One cannot replace in item $(3)$ of Theorem \ref{thm:fermat}  the field $\Q(i)$ with its subfield $\Q.$  Indeed, since $\Q$ is also a subfield of $\R$
it follows from  Corollary \ref{cor:tmsln}(2) that  $\SL(n,\Q)$ is (totally) minimal for every $n\in \N.$
\item If $p$ is a Fermat prime then every odd $n$ is coprime to $p-1.$ Hence,   $\SL(n,\F)$ and $\PSL(n,\F)$ are totally minimal for every subfield $\F$ of 
$\Q_p$, in view of Corollary \ref{c:p-adic}(1).
\een
\end{remark}


\vskip 0.7cm


		\bibliographystyle{plain}

\begin{thebibliography}{10}
			

		\bibitem{AHL16}	O. Ahl\'{e}n, \emph{Global Iwasawa-decomposition of $\SL(n, \mathbb A_{\mathbb Q}),$} 2016,  arXiv:1609.06621, 
		1--20. 
		
		\bibitem{AT} 
		A. Arhangel'skii, M. Tkachenko, \textit{Topological groups and related structures}, Atlantis Studies in Math. Ed.: J van Mill, World Scientific, 2008. 
		
		
		\bibitem{BG}
		U. Bader, T. Gelander, \emph{Equicontinuous actions of semisimple groups}, Groups, Geometry and Dynamics \textbf{11} (2017), 1003--1039.  
		
		\bibitem{Banakh}
		T. Banakh,  
		\textit{A quantitative generalization of Prodanov--Stoyanov Theorem on minimal Abelian topological groups}, Topology Appl. \textbf{271} (2020), 17 pp.
		
		
		\bibitem{B} B. Banaschewski, \textit{Minimal topological algebras}, Math. Ann. \textbf{211} (1974), 107--114.
		
	\bibitem{BT} 
		I. Ben Yaacov, T. Tsankov, 
	 \textit{Weakly almost periodic functions, model-theoretic stability, and minimality of topological groups}, Trans. Amer. Math. Soc. \textbf{368} (2016), no. 11,
	8267--8294.
		
			\bibitem{DS79} S. Dierolf, U. Schwanengel, \emph{Examples of locally compact non-compact minimal topological groups,} Pacific J. Math. \textbf{82} (1979),  349--355.
		
	
		\bibitem{DM10} D. Dikranjan,  M.  Megrelishvili, \emph{Relative minimality and co-minimality of subgroups in topological groups,}
		Topology Appl. \textbf{157} (2010), 62--76.
		
			\bibitem{DM14}	D. Dikranjan,  M.  Megrelishvili, \emph{Minimality Conditions in Topological Groups,}
			in: Recent Progress in General Topology III, 229--327, K.P. Hart, J. van Mill,
			P. Simon (Eds.), Springer, Atlantis Press, 2014.
			
				\bibitem{DP} D. Dikranjan,  
			Iv. Prodanov, \textit{Totally minimal topological groups}, Annuaire Univ. Sofia Fat. Math.
			M\'{e}c. \textbf{69} (1974/75), 5--11. 
			
		\bibitem{DPS89} D. Dikranjan, Iv. Prodanov,  L. Stoyanov, \emph{Topological groups: characters, dualities and minimal group
		topologies,} Pure and Appl. Math. \textbf{130}, Marcel Dekker, New York-Basel, 1989.
	
	\bibitem{Doitch}D. Do\"{\i}tchinov, \textit{Produits de groupes topologiques minimaux}, Bull. Sci. Math.  \textbf{97} (2)  (1972), 59--64.
	
	\bibitem{Duchesne} 
	B. Duchesne, \textit{The Polish topology of the isometry group of the infinite dimensional hyperbolic space}, arXiv:2005.12204, 2020. 
	
	\bibitem{Est} 
	W.T. van Est, \textit{Dense imbeddings of Lie groups}, Indag. Math. \textbf{13} (1951),  321--328.
	
	
	\bibitem{Gl}
	E. Glasner, \textit{The group $Aut(\mu)$ is Roelcke precompact}, Can. Math. Bull., \textbf{55} (2012), 297--302.
	
	\bibitem{Goto} 
	M. Goto, \textit{Absolutely closed Lie groups}, Math. Ann. \textbf{204} (1973),  337--341.
	
	
 \bibitem{IR} K. Ireland, M. Rosen, \emph{A Classical Introduction to Modern Number Theory}, second ed., Grad. Texts in Math., vol. \textbf{84}, Springer-Verlag, 1990.
	
	\bibitem{It} 
	K.B. Iturbe, \textit{Topological groups}, Final degree dissertation, Universidad del Pais Vasco, 2015. 
	
	
	\bibitem{Lorenz} 
	F. Lorenz, \textit{Algebra: Volume II: Fields with Structure, Algebras and Advanced Topics}, Springer, 2008. 
	
  \bibitem{MA} G.A. Margulis, \emph{Discrete Subgroups of Semisimple Lie Groups}, Springer-Verlag, 1990.
   

   \bibitem{Mayer} 
  M. Mayer, \textit{Asymptotics of matrix coefficients and closures
   of Fourier--Stieltjes algebras}, 
  J. of Functional Analysis \textbf{143} (1997), 42--54. 
  
		\bibitem{MEG95}	M. Megrelishvili, \emph{Group representations and construction of minimal topological groups}, 
		Topology Appl. \textbf{62} (1995), 1--19. 
	
		\bibitem{MEG98}	M. Megrelishvili, \emph{G-minimal topological groups,} in: Abelian Groups, Module Theory and Topology, in: Lecture Notes in Pure and Appl. Math., vol. \textbf{201}, Marcel Dekker, 1998, pp. 289--300.
	
	\bibitem{MEG04} 
	M. Megrelishvili, 
	\textit{Generalized Heisenberg groups and Shtern's question}, 
	Georgian Math. Journal  \textbf{11:4} (2004), 775--782.
	
	\bibitem{MEG08} 
	M. Megrelishvili, \textit{Every topological group is a group retract of a minimal group}, Topology Appl. \textbf{155} (2008), 2105--2127.
	
	
	\bibitem{Omori}
H. Omori,  \emph{Homomorphic images of Lie groups}, J. Math. Soc. Japan \textbf{18} (1966), 97--117.


\bibitem{PR} D. Prasad, A. Raghuram, \emph{Representation theory of $\GL(n)$ over non-Archimedean local fields}, School on Automorphic Forms on $\GL(n)$, 159--205, ICTP Lect. Notes, \textbf{21}, Abdus Salam Int. Cent. Theoret. Phys., Trieste, 2008.

\bibitem{P}
Iv. Prodanov,
\textit{Precompact minimal group topologies and $p$-adic numbers},
Annuaire Univ. Sofia Fac. Math. M\' ec.  \textbf{66} (1971/72), 249--266.


\bibitem{PS} 
Iv. Prodanov, L. Stoyanov, \textit{Every minimal abelian group is precompact}, C. R. Acad. Bulgare Sci. \textbf{37} (1984), 23--26.
	
	\bibitem{RS} 
D. Remus, L. Stoyanov, \emph{Complete minimal topological groups}, Topology Appl. \textbf{42} (1991), 57--69.


\bibitem{RO} D.J.S. Robinson, \emph{A Course in the Theory of Groups}, 2nd edn., Springer-Verlag, New York, 1996.

\bibitem{RoD} W. Roelcke,  
S. Dierolf, \emph{Uniform structures on topological groups and their quotients}, McGraw Hill, New York, 1981. 
	\bibitem{V} A.C.M. van Rooij, \emph{Non-Archimedean Functional Analysis}, Monographs and Textbooks in Pure and
Applied Math. \textbf{51}, Marcel Dekker, New York, 1978.



	
\bibitem{S71}
R.M. Stephenson, Jr., \textit{Minimal topological groups}, Math. Ann. \textbf{192} (1971), 193--195.
	
	
\bibitem{Stoyanov}
L. Stoyanov, \textit{Total minimality of the unitary groups}, Math. Z. \textbf{187} (1984), 273--283.

	
	
	\bibitem{Tits}  
	J. Tits, \textit{Reductive groups over local fields}, Proc.  Symp. Pure Math., in: Automorphic forms, representations and L-functions, ed. A. Borel \& W. Casselman,  \textbf{33} (1977),  29--71. 
	
	\bibitem{Usp}  
	V.V. Uspenskij, \textit{On subgroups of minimal topological groups}, Topology Appl. \textbf{155} (2008), 1580--1605.
	
	\bibitem{Warner-R}
	S. Warner, \emph{Topological rings}, Math. Studies \textbf{178} (1993). 
	
	\bibitem{XDSD} W. Xi, D. Dikranjan, M. Shlossberg, D. Toller, 
	\emph{Hereditarily minimal topological groups}, 
	 Forum Math. 
	 \textbf{31:3} (2019), 619--646.
	\end{thebibliography}


\end{document}